\documentclass[12pt]{amsart}

\usepackage{latexsym, amssymb, amsmath,ulem,soul,esint}

\usepackage{amsfonts, graphicx}
\usepackage{graphicx,color}
\newcommand{\be}{\begin{equation}}
\newcommand{\ee}{\end{equation}}
\newcommand{\beq}{\begin{eqnarray}}
\newcommand{\eeq}{\end{eqnarray}}

\newtheorem{thm}{Theorem}[section]
\newtheorem{conj}{Conjecture}[section]

\newtheorem{lma}{Lemma}[section]
\newtheorem{prop}{Proposition}[section]

\newtheorem{defn}{Definition}[section]

\newtheorem{claim}{Claim}[section]
\theoremstyle{remark}
\newtheorem{rem}{Remark}[section]
\numberwithin{equation}{section}

\def\be{\begin{equation}}
\def\ee{\end{equation}}
\def\bee{\begin{equation*}}
\def\eee{\end{equation*}}

\newcommand{\ddb}{\partial \ov{\partial}}

\newcommand{\vp}{\varphi}

\def\K{K\"ahler }

\def\arccot{\mathrm{arccot}}

\def\a{{\alpha}}
\def\b{{\beta}}

\def\H{\mathcal{H}}

\def\ddb{\sqrt{-1}\partial\bar\partial}

\def\Re{\mathrm{Re}}
\def\Im{\mathrm{Im}}

\begin{document}

\title{Hypercritical deformed Hermitian-Yang-Mills equation revisited}

\author{Jianchun Chu}
\address[Jianchun Chu]{School of Mathematical Sciences, Peking University, Yiheyuan Road 5, Beijing, P.R.China, 100871}
\email{jianchunchu@math.pku.edu.cn}

\author{Man-Chun Lee}
\address[Man-Chun Lee]{Department of Mathematics, The Chinese University of Hong Kong, Shatin, N.T., Hong Kong}
\email{mclee@math.cuhk.edu.hk}

\subjclass[2020]{Primary: 53C55; Secondary: 35A01.}

\date{\today}

\begin{abstract}
In this paper, we study the hypercritical deformed Hermitian-Yang-Mills equation on compact K\"ahler manifolds and resolve two conjectures of Collins-Yau \cite{CollinsYau2018}.
\end{abstract}

\keywords{Deformed Hermitian-Yang-Mills equation, Hypercritical phase}

\maketitle

\section{Introduction}

Let $(X^n,\omega)$ be a compact \K manifold and $\a$ be a closed real $(1,1)$ form on $X$ so that $\int_X (\a+\sqrt{-1}\omega)^n \neq 0$ and therefore we might write
\begin{equation}\label{theta 0}
\int_X (\a+\sqrt{-1}\omega)^n=\mathbb{R}_{>0}\cdot e^{\sqrt{-1}\theta_0}
\end{equation}
for some $e^{\sqrt{-1}\theta_0}\in \mathbb{S}^1$. In particular, the angle $\theta_0$ is well-defined modulo $2\pi$. The deformed Hermitian-Yang-Mills (dHYM) equation seeks for $\varphi\in C^\infty(X)$ such that $\a_\varphi=\a+\ddb \varphi$ satisfies
\begin{equation}\label{dHYM-equ}
\Im \left(e^{-\sqrt{-1}\theta_0} (\a_\varphi+\sqrt{-1}\omega)^n \right)=0.
\end{equation}

The dHYM equation first appeared in \cite{LeungYauZaslow2001} from the mathematical side drawing from the physics literature \cite{MMMS00} which is corresponding to the special Lagrangian equation under the setting of the Strominger-Yau-Zaslow mirror symmetry \cite{SYZ96}.

One of the main topic in the study of dHYM equation is to characterize the solvability in terms of certain algebraic conditions on the class $[\a]$.  In \cite[Conjecture 1.4]{CollinsJacobYau2020}, Collins-Jacob-Yau predicted that the existence of solution to the supercritical dHYM equation is equivalent to a stability condition in terms of holomorphic intersection numbers for any irreducible subvarieties $V\subset X$, modeled on the Nakai-Moishezon criterion,  and confirmed it for complex surfaces.  In \cite{ChuLeeTakahashi2021}, the authors and Takahashi confirmed the conjecture in the projective case building on the works of Chen \cite{Chen2021} and Song \cite{Song2020}, see also \cite{DatarPingali2021,JacobSheu2020}.

On the other hand,  motivated by the GIT (Geometric Invariant Theory) approach for special Lagrangian \cite{Thomas2001,Solomon2013},  Collins-Yau \cite{CollinsYau2018} proposed to study the dHYM equation using the space $\mathcal{H}_\omega$ of almost calibrated $(1,1)$ forms in the class $[\a]$:
\begin{equation}
\mathcal{H}_\omega=\left\{\varphi\in C^\infty(X):  \Re\left(e^{-\sqrt{-1}\theta_0}(\a_\varphi+\sqrt{-1}\omega)^n \right)>0 \right\}.
\end{equation}

The space $\mathcal{H}_\omega$ is a (possibly empty) open subset of the space of smooth, real valued functions on $X$.  By studying the geodesic and functional on $\mathcal{H}$,  Collins-Yau \cite{CollinsYau2018} discovered a number of algebraic obstruction to the dHYM solution in the hypercritical phase.  We refer interested readers to the survey article \cite{CollinsShi2020} for a comprehensive discussion.

When $\mathcal{H}_\omega\neq \emptyset$,  a maximum principle shows that
\begin{equation}
\mathcal{H}_\omega=\left\{\varphi\in C^\infty(X): |Q_\omega(\a_\varphi)-\beta|<\frac\pi2\right\}
\end{equation}
where $Q_\omega(\a_\varphi)$ is the special Lagrangian operator defined by \eqref{lag-operator} and $\beta$ is some lift of $\theta_0$ from $\mathbb{R}/2\pi\mathbb{Z}$ to $(0,n\pi)$. The lift $\beta$ is usually referred to the analytic lifted angle. To determine the non-emptiness of $\mathcal{H}_\omega$ using algebraic information of $[\a]$,  Collins-Yau \cite[Section 8]{CollinsYau2018} introduced an algebraic approach in determining the lifted angle, see Definition~\ref{Z theta vp definitions}. Particularly, using a Chern number inequality in \cite{CollinsXieYau2018}, it was shown that the algebraic lifted angle coincides with the analytic lifted angle in three dimensional whenever a supercritical dHYM solution exists.  Moreover, the following was shown.

\begin{prop}[Proposition 8.4 in \cite{CollinsYau2018}]\label{proposition-OF-CY18}
Suppose $(X^3,\omega)$ is a compact three-dimensional \K manifold and $[\a]\in H^{1,1}(X,\mathbb{R})$.  If the dHYM equation admits a solution with $\theta\in (0,\frac\pi2)$\footnote{The convention taken here is slightly different from that in \cite{CollinsYau2018}. The range of $\theta\in (0,\frac\pi2)$ is equivalent to $\hat\theta\in(\pi,\frac{3\pi}{2})$ there.} then the followings hold:
\begin{enumerate}\setlength{\itemsep}{1mm}
\item[(i)] The Chern number satisfies
$$\left(\int_X \a^3 \right)\left(\int_X \omega^3 \right)< 9 \left(\int_X \a \wedge \omega^2 \right)\left(\int_X \a^2 \wedge \omega \right),$$
in particular the algebraic lifted angle $\hat\theta_X([\a])$ is well-defined;
\item[(ii)] $\mathrm{Im}(Z_X([\a])>0$ and $\varphi_X([\a])\in (\frac\pi2,\pi)$;
\item[(iii)] For any irreducible subvariety $V\subsetneq X$,
\[
\mathrm{Im}(Z_{V}([\alpha])) > 0 ,\quad
\vp_{V}([\alpha]) > \vp_{X}([\alpha]).
\]
\end{enumerate}
\end{prop}

It is conjectured that the converse should also hold, see \cite[Conjecture 8.5]{CollinsYau2018}. In this work,  we give an affirmative answer to this question.
\begin{thm}\label{main-int}
The converse of Proposition~\ref{proposition-OF-CY18} is true.
\end{thm}

The resolution of the conjecture is based on a Nakai-Moishezon type criterion proved by the authors and Takahashi \cite{ChuLeeTakahashi2021}. The most crucial observation is to show that the assumptions (i)-(iii) indeed give rise to the K\"ahlerity of $[\a]$ and a stability in terms of intersection number of subvariety in $X$.

In \cite{CollinsYau2018},  it is also conjectured that the non-emptiness of $\mathcal{H}_{\omega}$ is equivalent to certain Nakai-Moishezon type criterion.
\begin{conj}[Conjecture 8.7 in \cite{CollinsYau2018}]\label{CY conjecture H}
The followings are equivalent:
\begin{enumerate}\setlength{\itemsep}{1mm}
\item[(A)] The space $\H_{\omega}$ is non-empty and $[\alpha]$ has hypercritical phase;
\item[(B)] For any irreducible subvariety $V\subset X$, $\Im(Z_{V,[\alpha]})>0$.
\end{enumerate}
\end{conj}
The implication $(A)\Longrightarrow (B)$ has been established in \cite[Corollary 8.6]{CollinsYau2018}.  Though an example in blow-up of $\mathbb{CP}^2$ at one point, we find that the converse is not necessarily true.
\begin{prop}\label{counterexample}
On $X=\mathrm{Bl}_{p}(\mathbb{CP}^2)$,  there exist \K class $[\omega]$ and $[\a]\in H^{1,1}(X,\mathbb{R})$ such that (B) in Conjecture~\ref{CY conjecture H} holds but $\mathcal{H}_{\omega}=\emptyset$.
\end{prop}

In contrast, we can provide an alternative criteria of $\mathcal{H}_{\omega}\neq \emptyset$ in terms of stability condition on holomorphic intersection numbers for any irreducible subvariety $V\subset X$ based on the work in \cite{ChuLeeTakahashi2021}, see Theorem~\ref{CS question answer} and Remark~\ref{CS question answer-2}.

\medskip

The paper is organized as follows: In Section 2, we will collect some preliminaries and notations that will be used throughout this work. In Section 3, we will give the proof of Theorem~\ref{main-int}. In Section 4, we will prove Proposition~\ref{counterexample} which gives a counter-example of Conjecture~\ref{CY conjecture H}.  In Section 5, we will discuss a criteria of $\mathcal{H}_\omega\neq \emptyset$.

\medskip

{\it Acknowledgement}:
J. Chu was partially supported by Fundamental Research Funds for the Central Universities (No. 7100603624).

\section{Preliminaries and notations}
In this section, we will introduce the necessary notations in this work.  The ultimate goal is to understand the existence of solution to the dHYM equation \eqref{dHYM-equ}. Locally, if we choose a local holomorphic coordinate around $p\in X$ so that $\a_\varphi(p)$ is diagonal with respect to $\omega(p)$ with eigenvalues $\lambda_i$, then
\begin{equation}
\begin{split}
 \frac{(\a_\varphi+\sqrt{-1}\omega)^n}{\omega^n}
 &=\sqrt{\prod_{i=1}^n (1+\lambda_i^2)} \cdot e^{\sqrt{-1}\sum_{i=1}^n \arccot(\lambda_i)}.
 \end{split}
\end{equation}
In this way, we define the Lagrangian phase operator\footnote{In the literature, it is sometime convenient to consider the integral $\int_X (\omega+\sqrt{-1}\a)^n$ instead and the corresponding Lagrangian phase operator will be defined as $\hat Q_\omega(\a_{\vp})=\sum_{i=1}^n\arctan(\lambda_i)$ instead.} as
\begin{equation}\label{lag-operator}
Q_\omega(\a_\varphi)=\sum_{i=1}^n \arccot(\lambda_i).
\end{equation}
In other words, the dHYM equation seeks for $\varphi\in C^\infty(X)$ so that
\begin{equation}
Q_\omega(\a_\varphi)=\theta_0 \ \ {\mathrm{mod}} \; 2\pi.
\end{equation}
where $e^{\sqrt{-1}\theta_0}$ is a cohomological constant determined by the class $[\omega]$ and $[\a]$.

\medskip

The space of almost calibrated $(1,1)$ forms in the class $[\a]$ is given by
\begin{equation}
\mathcal{H}_\omega=\left\{\varphi\in C^\infty(X):  \Re\left(e^{-\sqrt{-1}\theta_0}(\a_\varphi+\sqrt{-1}\omega)^n \right)>0 \right\}.
\end{equation}
In general, the space $\mathcal{H}_{\omega}$ depends also on the representative $\omega$ of $[\omega]$.

Since $\theta_0$ is a-priori only defined in $\mathbb{R}/2\pi\mathbb{Z}$, $\mathcal{H}_{\omega}$ will be a disjoint union of branches. It is an application of maximum principle \cite{CollinsXieYau2018} that if $\mathcal{H}_{\omega}\neq \emptyset$, then we have
\begin{equation}
\mathcal{H}_{\omega}=\left\{\varphi\in C^\infty(X): |Q_\omega(\a_\varphi)-\beta|<\frac\pi2\right\}
\end{equation}
for an unique $\beta\in (0,n\pi)$ so that $\b=\theta_0 \; (\mathrm{mod}\; 2\pi)$. The lift $\beta$ is usually referred to the analytic lifted angle. For notational convenience, if $\mathcal{H}_{\omega}\neq\emptyset$, we will use $\theta_0$ to denote this uniquely defined lifted phase $\b$. And thus, the dHYM equation can be rewritten as
\begin{equation}
Q_\omega(\a_\varphi)=\theta_0 \in \mathbb{R}.
\end{equation}

When the lifted phase $\theta_0\in (0,\frac\pi2)$, we say that $[\a]$ has the hypercritical phase, while if $\theta_0\in (0,\pi)$, $[\a]$ is said to have supercritical phase. When the lifted phase lies inside the region of supercritical phase, the dHYM equation is known to be well-behaved in the analytic point of view. It is therefore important to determine the lifted angle. In \cite{CollinsYau2018}, Collins-Yau proposed a purely algebraic approach to determine the lift. They introduced the following.

\begin{defn}\label{Z theta vp definitions}
Let $(X,\omega)$ be a compact $n$-dimensional K\"ahler manifold. For $[\alpha]\in H^{1,1}(X,\mathbb{R})$ and $p$-dimensional irreducible subvariety $V\subset X$, define
\begin{equation*}
\left\{
\begin{array}{ll}
Z_{V,[\alpha]}(t) =\displaystyle -\int_V e^{-\sqrt{-1}(t\omega+\sqrt{-1}\a)}= -\frac{(-\sqrt{-1})^{p}}{p!}\int_{V}(t\omega+\sqrt{-1}\alpha)^{p};\\[5mm]
Z_{V}([\alpha]) = Z_{V,[\alpha]}(1)
\end{array}
\right.
\end{equation*}
for $t\in [1,+\infty]$. Suppose that $Z_{V,[\alpha]}(t)\in\mathbb{C}^{*}$ for all $t\in[1,+\infty]$.
\begin{enumerate}\setlength{\itemsep}{1mm}
\item[(i)] The algebraic lifted angle $\hat{\theta}_{V}([\alpha])$ is defined as the winding angle of the curve $Z_{V,[\alpha]}(t)$ as $t$ runs from $+\infty$ to $1$.
\item[(ii)] The slicing angle $\vp_{V}([\alpha])$ is defined as
\[
\vp_{V}([\alpha]) = \hat{\theta}_{V}([\alpha])-(p-2)\cdot\frac{\pi}{2}.
\]
\end{enumerate}
\end{defn}

\section{Proof of Theorem~\ref{intro-mainTHM}}

In this section, we will establish the characterization of existence of hypercritical dHYM solution in three dimension, namely Theorem~\ref{intro-mainTHM}. We start with some preparation lemmas.

\begin{lma}\label{lma:range-OF-varphi}
Under the assumption (i), (ii) and (iii) in Theorem~\ref{intro-mainTHM}, the following holds.
For any proper $p$-dimensional irreducible subvariety $V\subsetneq X$, we have
\begin{equation}
\frac\pi2<\varphi_X([\a])<\varphi_V([\a])<\pi.
\end{equation}
Moreover, $\displaystyle Z_{V}([\a])\in \mathbb{R}_{>0} \cdot e^{\sqrt{-1}\varphi_V([\a])}$.
\end{lma}
\begin{proof}
The first two inequalities follows from assumption (ii) and (iii). It suffices to show $\varphi_V([\a])<\pi$.  Indeed, this follows from the following simple observation.  By definition, the algebraic lifted angle $\hat\theta_V([\a])$ is given by
\begin{equation}
\lim_{t\to+\infty}\frac{Z_{V,[\a]}(1)}{Z_{V,[\a]}(t)}\in \mathbb{R}_{>0} \cdot e^{\sqrt{-1}\hat\theta_V([\a])}.
\end{equation}
Together with the fact that as $t\to +\infty$,
\begin{equation}
Z_{V,[\a]}(t)\approx e^{-\sqrt{-1}(p-2) \frac\pi2}\cdot \frac{t^p}{p!} \int_V \omega^p,
\end{equation}
this shows that
\begin{equation}\label{Z varphi}
Z_{V}([\a]) = Z_{V,[\a]}(1) \in \mathbb{R}_{>0} \cdot e^{\sqrt{-1}\varphi_V([\a])}.
\end{equation}

When $p=1$,
\begin{equation}
Z_{V,[\alpha]}(t) = -\int_{V}\alpha+\sqrt{-1}t\int_{V}\omega.
\end{equation}
For $t\in[1,+\infty]$, it is clear that
\begin{equation}
\mathrm{Im}(Z_{V,[\alpha]}(t)) > 0.
\end{equation}
This implies $\hat{\theta}_{V}([\alpha])\in(0,\pi)$ and
\begin{equation}
\vp_{V}([\alpha]) = \hat{\theta}_{V}([\alpha]) + \frac{\pi}{2} \in \left(\frac{\pi}{2},\frac{3\pi}{2}\right).
\end{equation}
Combining this with \eqref{Z varphi} and $\Im(Z_{V,[\alpha]}(1))=\int_{V}\omega>0$, we see that $\vp_{V}([\alpha])<\pi$.

When $p=2$,
\begin{equation}
Z_{V,[\alpha]}(t) = \frac{1}{2}\int_{V}(t^{2}\omega^{2}-\alpha^{2})+\sqrt{-1}t\int_{V}\alpha\wedge\omega.
\end{equation}
By assumption (iii), we obtain $\int_{V}\alpha\wedge\omega>0$, and then for $t\in[1,+\infty]$,
\begin{equation}
\mathrm{Im}(Z_{V,[\alpha]}(t)) > 0.
\end{equation}
This implies $\hat{\theta}_{V}([\alpha])\in(0,\pi)$ and $\vp_{V}([\alpha]) = \hat{\theta}_{V}([\alpha])<\pi$.
\end{proof}

Next, we wish to show that $[\a]\in H^{1,1}(X,\mathbb{R})$ is in fact a \K class.  This is analogous to the K\"ahlerity of $[\a]$ if it is a sub-solution in the hypercritical phase.
\begin{lma}\label{lma-Kahlerclass}
Under the assumption (i), (ii) and (iii) in Theorem~\ref{intro-mainTHM}, $[\a]\in H^{1,1}(X,\mathbb{R})$ is a \K class.
\end{lma}
\begin{proof}
By \cite[Theorem 4.2]{DemaillyPaun2004}, it suffices to show that for any $p$-dimensional irreducible subvariety $V\subset X$ and $k=1,2,\ldots,p$, we have
\begin{equation}
\int_V \a^k\wedge \omega^{p-k}>0.
\end{equation}

When $p=1$,   Lemma~\ref{lma:range-OF-varphi} implies that $\varphi_V([\a])\in (\frac\pi2,\pi)$ and hence $\Re\left( Z_{V}([\a])\right)<0$.  Since
\begin{equation}
 Z_{V}([\a])=\sqrt{-1}\cdot \left(\int_V \omega+\sqrt{-1}\a \right),
\end{equation}
this gives $\displaystyle\int_V \a >0$.

\smallskip

When $p=2$,
\begin{equation}
\begin{split}
2\cdot  Z_{V}([\a])&=   \int_V (\omega+\sqrt{-1}\a)^2\\
 &= \left(\int_V \omega^2-\a^2 +\sqrt{-1}\int_V 2\a\wedge \omega \right).
\end{split}
\end{equation}
Hence,  Lemma~\ref{lma:range-OF-varphi} implies
\begin{equation}\label{Kahler-equ-1}
\int_V \a\wedge \omega >0 \quad\text{and}\quad  \int_V \a^2>\int_V \omega^2>0.
\end{equation}

\smallskip

When $p=3$,  $V=X$ and hence
\begin{equation}
\begin{split}
6\cdot Z_X([\a])&=-\sqrt{-1} \int_X (\omega+\sqrt{-1}\a)^3\\
&=-\sqrt{-1} \left(\int_X \omega^3-3\a^2\wedge\omega+ \sqrt{-1}\int_X3\a\wedge\omega^2-\a^3 \right)\\
&=\left(\int_X 3\a\wedge\omega^2-\a^3 \right)+\sqrt{-1}\left(\int_X 3\a^{2}\wedge\omega-\omega^3 \right).
\end{split}
\end{equation}
Since $\varphi_X([\a])\in (\frac\pi2,\pi)$,  we have
\begin{equation}\label{Kahler-equ-3}
\int_X 3\a\wedge\omega^2  < \int_X \a^3 \quad\text{and}\quad \int_X 3\a^2\wedge\omega > \int_X \omega^3>0.
\end{equation}
Therefore, it remains to show that $\int_X \omega^2\wedge \a>0$. Using the assumption (i) on the Chern number,
\begin{equation}
\begin{split}
3\left(\int_X\a\wedge  \omega^2 \right) \left(\int_X \omega^3\right)
&<\left(\int_X \a^3 \right) \left(\int_X \omega^3\right) \\
&<9 \left(\int_X \a \wedge \omega^2 \right)\left(\int_X \a^2 \wedge \omega \right).
\end{split}
\end{equation}
The integral $\int_X \a \wedge \omega^2$ is clearly non-zero. If it is negative, we will have
\begin{equation}
\int_X \omega^3 > \int_X 3\a^2\wedge\omega > \int_X \omega^3,
\end{equation}
which is impossible. In conclusion, we have
\begin{equation}
\int_X \a\wedge \omega^2, \quad \int_X \a^2 \wedge \omega,\quad \int_X \a^3>0.
\end{equation}
This completes the proof.
\end{proof}

Next, we show that the class $[\a]$ will satisfy a kind of intersection number. This is in the same spirit as the numerical criterion of the \K class proved by Demailly-P\u{a}un \cite{DemaillyPaun2004}.
\begin{lma}\label{lma:-intersection-sub}Under the assumption (i), (ii) and (iii) in Theorem~\ref{intro-mainTHM}, the following holds.
There is $\theta_0\in (0,\frac\pi2)$ such that
\begin{equation}
\int_X \Re \left(\a+\sqrt{-1}\omega \right)^n -\cot\theta_0\cdot  \Im \left(\a+\sqrt{-1}\omega \right)^n=0.
\end{equation}
And  for all $p$-dimensional irreducible subvariety $V\subsetneq X$, we have
\begin{equation}
\int_V \Re \left(\a+\sqrt{-1}\omega \right)^p -\cot\theta_0\cdot  \Im \left(\a+\sqrt{-1}\omega \right)^p>0.
\end{equation}
\end{lma}
\begin{proof}
Clearly, $\theta_0$ is determined by the class of $\omega$ and $\a$.  We first show that $\theta_0$ is in the desired range.   Direct computation and the computation in the proof of Lemma~\ref{lma-Kahlerclass} shows that
\begin{equation}\label{Kahler-equ-4}
\left\{
\begin{array}{ll}
\displaystyle\int_X \Re(\a+\sqrt{-1}\omega)^3 = \int_X \a^3-3 \a \wedge \omega^2>0; \\[5mm]
\displaystyle\int_X \Im(\a+\sqrt{-1}\omega)^3 = \int_X 3\a^2\wedge \omega- \omega^3>0.
\end{array}
\right.
\end{equation}
If $\theta_0\in (0,2\pi)$ is chosen so that
\begin{equation}
\int_X \Re(\a+\sqrt{-1}\omega)^3  = \cot \theta_0 \cdot \int_X \Im(\a+\sqrt{-1}\omega)^3,
\end{equation}
then assumption (ii) forces $\theta_0\in (0,\frac\pi2)$. This proves the first assertion.

It remains to consider the integral on the irreducible subvariety $V\subsetneq X$.  We first  relate $Z_{V}([\a])$ with $\mathrm{Arg}\left(\int_V (\a+\sqrt{-1}\omega)^p\right)$. For any $p$-dimensional irreducible subvariety $V\subset X$, by using $\varphi_X([\a])<\varphi_V([\a])$,
\begin{equation}
\begin{split}
p!\cdot Z_{V}([\a])&=-(-\sqrt{-1})^p \cdot \int_V (\omega+\sqrt{-1}\a)^p\\
&=- \int_V (\a-\sqrt{-1}\omega)^p\\
&=e^{\sqrt{-1}\pi}\cdot \overline{\int_V (\a+\sqrt{-1}\omega)^{p}}.
\end{split}
\end{equation}
Since $\varphi_V([\a])\in (\frac\pi2,\pi)$ by Lemma~\ref{lma:range-OF-varphi},
\begin{equation}
\mathrm{Arg}\left(\int_V (\a+\sqrt{-1}\omega)^p\right)=\pi-\varphi_V([\a]).
\end{equation}
In particular,  $\theta_0=\pi-\varphi_X([\a])$ and therefore for any irreducible subvariety $V\subsetneq X$,
\begin{equation}\label{Kahler-equ-2}
0<\mathrm{Arg}\left(\int_V (\a+\sqrt{-1}\omega)^p\right)<\theta_0<\frac\pi2.
\end{equation}
We complete the proof.
\end{proof}

We remark here that if \cite[Conjecture 1.4]{CollinsJacobYau2020} holds, then the main result will follow from Lemma~\ref{lma:-intersection-sub}.  Now we are ready to prove the main theorem.

\begin{thm}\label{intro-mainTHM}
Suppose $(X^3,\omega)$ is a compact three-dimensional \K manifold and $[\a]\in H^{1,1}(X,\mathbb{R})$. Then the dHYM equation admits a solution with $\theta\in (0,\frac\pi2)$ if and only if the followings hold:
\begin{enumerate}\setlength{\itemsep}{1mm}
\item[(i)] The Chern number satisfies
$$\left(\int_X \a^3 \right)\left(\int_X \omega^3 \right)< 9 \left(\int_X \a \wedge \omega^2 \right)\left(\int_X \a^2 \wedge \omega \right),$$
in particular the algebraic lifted angle $\hat\theta_X([\a])$ is well-defined;
\item[(ii)] $\mathrm{Im}(Z_X([\a])>0$ and $\varphi_X([\a])\in (\frac\pi2,\pi)$;
\item[(iii)] For any irreducible subvariety $V\subsetneq X$,
\[
\mathrm{Im}(Z_{V}([\alpha])) > 0 ,\quad
\vp_{V}([\alpha]) > \vp_{X}([\alpha]).
\]
\end{enumerate}
\end{thm}
\begin{proof}
We begin by noting that if there exists a dHYM solution with lifted angle $\theta_0\in (0,\frac\pi2)$, then
\begin{equation}
\sum_{i=1}^3 \arctan\lambda_i =\hat\theta_0\in \left(\pi,\frac{3\pi}{2} \right).
\end{equation}
Then (i)-(iii) follows from the same argument as in \cite[Proposition 8.4]{CollinsYau2018}. In \cite{CollinsYau2018}, $[\a]$ is assumed to be $c_1(L)$ for some line bundle $L$. It is clear from the proof that $[\a]\in H^{1,1}(X,\mathbb{R})$ suffices, see also \cite{CollinsXieYau2018}.

It remains to prove the existence of dHYM solution under assumption (i)-(iii).  We fix $\theta_0\in (0,\frac\pi2)$ from Lemma~\ref{lma:-intersection-sub}.
\begin{claim}
For any $k=1,2,3$, we have
\begin{equation}
\int_X \left( \Re(\a+\sqrt{-1}\omega )^k-\cot\theta_0 \cdot \Im(\a+\sqrt{-1}\omega)^k \right)\wedge \a^{3-k}\geq 0;
\end{equation}
and for any $p$-dimensional irreducible subvariety $V\subsetneq X$ and $k=1,2,\ldots,p$,
\begin{equation}
\int_V \left( \Re(\a+\sqrt{-1}\omega )^k-\cot\theta_0 \cdot \Im(\a+\sqrt{-1}\omega)^k \right)\wedge \a^{p-k}> 0.
\end{equation}
\end{claim}

\begin{proof}[Proof of Claim]

By Lemma~\ref{lma:-intersection-sub}, it remains to consider the following cases: $(p,k)=(2,1)$, $(3,2)$ and $(3,1)$.

\smallskip

When $(p,k)=(2,1)$, we have from \eqref{Kahler-equ-2} and $[\a]>0$ that
\begin{equation}
\begin{split}
\int_V \a^2 -\omega^2 & = \cot\left(\mathrm{Arg} \Big(\int_V (\a+\sqrt{-1}\omega)^2 \Big)\right) \cdot 2\int_V \a\wedge \omega\\
&> \cot\theta_0 \cdot 2\int_V \a\wedge \omega.
\end{split}
\end{equation}
Therefore,
\begin{equation}
\begin{split}
&\quad \int_V \left[\Re(\a+\sqrt{-1}\omega) -\cot \theta_0 \cdot \Im(\a+\sqrt{-1}\omega)\right]\wedge \a\\
&=\int_V \a^2 -\cot\theta_0 \cdot \a\wedge \omega>\int_V \omega^2 +\cot\theta_0 \cdot \a\wedge \omega>0.
\end{split}
\end{equation}

\smallskip

We proceed to consider $p=3$. For notational convenience, we denote
\begin{equation}
a_i=\int_X \a^i\wedge \omega^{3-i},\quad\text{for}\; i=0,1,2,3.
\end{equation}
Then the assumption (i), $\theta_0\in (0,\frac\pi2)$ and K\"ahlerity of $[\a]$ can be reduced to
\begin{equation}
\left\{
\begin{array}{ll}
a_0a_3 < 9 a_1 a_2;\\[3mm]
0<3a_1< a_3;\\[3mm]
0<3a_2>a_0;\\[2mm]
\cot\theta_0= \displaystyle\frac{a_3-3a_1}{3a_2-a_0}\in \mathbb{R}_{>0}.
\end{array}
\right.
\end{equation}
If $k=1$,
\begin{equation}
\begin{split}
&\quad \int_X \left(\Re(\a+\sqrt{-1}\omega)-\cot\theta_0 \cdot \Im (\a+\sqrt{-1}\omega) \right)\wedge \a^2\\
&=a_3 -\frac{a_3-3a_1}{3a_2-a_0}\cdot a_2\\[2mm]
&=\frac{2a_2 a_3-a_3a_0+3a_1a_2}{3a_2-a_0}\\[1mm]
&>\frac{2a_2 (a_3-3a_1)}{3a_2-a_0}>0.
\end{split}
\end{equation}
If $k=2$,
\begin{equation}
\begin{split}
&\quad \int_X \left(\Re(\a+\sqrt{-1}\omega)^2-\cot\theta_0 \cdot \Im (\a+\sqrt{-1}\omega)^2 \right)\wedge \a\\
&=\int_X \left[ (\a^2 -\omega^2) -\cot\theta_0 \cdot (2\a\wedge \omega) \right]\wedge \a\\
&=(a_3-a_1)-\left(\frac{a_3-3a_1}{3a_2-a_0}\right)\cdot 2a_2\\[2mm]
&=\frac{3a_1a_2+a_2a_3+a_1a_0-a_0a_3}{3a_2-a_0}\\[1mm]
&>\frac{1}{3a_2-a_0}\left(-\frac23 a_0a_3 +a_1a_0+ \frac{a_0a_3^2}{9a_1} \right)\\
&=\frac{a_0 a_1}{3a_2-a_0} \left(\frac{a_3}{3a_1}-1\right)^2\geq 0.
\end{split}
\end{equation}
\end{proof}

Since $[\a]$ is a \K class by Lemma~\ref{lma-Kahlerclass}, the existence of dHYM solution with hypercritical phase follows from the Claim and \cite[Corollary 1.4]{ChuLeeTakahashi2021}. This completes the proof.
\end{proof}

\section{Counter-example on Blow-up of $\mathbb{CP}^2$}
In this section,  we will prove Proposition \ref{counterexample}. Let $X$ be the blow-up of $\mathbb{CP}^{2}$ at one point, $H$ be the pull-back of the hyperplane divisor, and $E$ be the exceptional divisor. It is well-known that
\begin{equation}
H^{2} = 1, \quad
E^{2} = -1, \quad
H \cdot E = 0,
\end{equation}
and $a[H]-[E]$ is K\"ahler when $a>1$. Now we choose
\begin{equation}
[\omega] = 2[H]-[E], \quad
[\alpha] = 6[H]+[E].
\end{equation}

\smallskip

\begin{proof}[Proof of Proposition \ref{counterexample}]
By direct calculation,
\begin{equation}\label{counterexample eqn 1}
\int_{X}(\alpha+\sqrt{-1}\omega)^{2}
= \int_{X}(\alpha^{2}-\omega^{2})+2\sqrt{-1}\int_{X}\alpha\wedge\omega
=  32+26\sqrt{-1}.
\end{equation}
Then the complex number $\int_{X}(\alpha+\sqrt{-1}\omega)^{2}$ lies in the first quadrant of $\mathbb{C}$. For any $1$-dimensional irreducible subvariety $V\subset X$,
\begin{equation}
Z_{V}([\alpha]) = -\int_{V}\alpha+\sqrt{-1}\int_{V}\omega,
\end{equation}
and hence $\Im(Z_{V}([\alpha]))>0$. When $V=X$,
\begin{equation}
Z_{X}([\alpha]) = -\frac{1}{2}\int_{V}(\alpha^{2}-\omega^{2})+\sqrt{-1}\int_{X}\alpha\wedge\omega
= -16+13\sqrt{-1}.
\end{equation}
Now we show that $\H_{\omega}$ is empty. If $\H_{\omega}\neq\emptyset$, then by $\dim X=2$, there exists $\theta_{0}\in(0,n\pi)=(0,2\pi)$ such that
\begin{equation}
\int_{X}(\alpha+\sqrt{-1}\omega)^{2} \in \mathbb{R}_{>0} \cdot e^{\sqrt{-1}\theta_{0}}
\end{equation}
and
\begin{equation}
\H_{\omega} = \left\{ \vp\in C^{\infty}(X)~\big|~|Q(\alpha_{\vp})-\theta_{0}|<\frac{\pi}{2} \right\}.
\end{equation}
By \eqref{counterexample eqn 1}, we see that $\theta_{0}\in(0,\frac{\pi}{2})$ and $\tan\theta_{0}=\frac{13}{16}$. For $\vp\in\H_{\omega}$, let $\lambda_{1}$ and $\lambda_{2}$ be the eigenvalues of $\alpha_{\vp}$ with respect to $\omega$. It then follows that for $i=1,2$,
\begin{equation}
0 < \arccot(\lambda_{i}) < \arccot(\lambda_{1})+\arccot(\lambda_{2}) = Q_{\omega}(\alpha_{\vp}) < \theta_{0}+\frac{\pi}{2} < \pi
\end{equation}
and so $\lambda_{i}>-\tan\theta_{0}\cdot\omega$. This implies $\alpha_{\vp}+\tan\theta_{0}\cdot\omega>0$. In particular,
\begin{equation}
0 < \int_{E}\alpha_{\vp}+\tan\theta_{0}\cdot\omega = \int_{X}(\alpha+\tan\theta_{0}\cdot\omega)\wedge[E] = -1+\tan\theta_{0} = -\frac{3}{16},
\end{equation}
which is impossible.
\end{proof}

\section{Non-emptiness of $\mathcal{H}_\omega$ under test family condition}

In \cite{ChuLeeTakahashi2021}, it is proved that the dHYM equation admits a supercritical phase solution if and only if the triple $(X,\omega,\a)$ is stable along some test family. In this section, we find that a similar type of stability also give rise to non-emptiness of the space $\mathcal{H}_{\omega}$ of almost calibrated $(1,1)$ forms.  We start by recalling the concept of test family defined by Chen \cite{Chen2021}.
\begin{defn}
A family of $(1,1)$ forms $\a_t$, $t\in [0,+\infty)$ is said to be a $\theta$-test family (emanating from a real $(1, 1)$ form $\a$) if
\begin{enumerate}\setlength{\itemsep}{1mm}
\item[(a)] $\a_0=\a$;
\item[(b)] $\a_t>\a_s$ if $t>s$;
\item[(c)] there exists $T\geq 0$ such that $\a_T>\cot \left(\frac{\theta}{n} \right)\cdot\omega$ for all $t>T$.
\end{enumerate}
\end{defn}

Now we are ready to state the criteria in terms of test family.
\begin{thm}\label{CS question answer}
Suppose \eqref{theta 0} holds and there exists a $(\theta_0+\frac\pi2)$-test family $\alpha_{t}$ for some $\theta_0\in (0,\frac\pi2)$ such that for any $p$-dimensional subvariety $V\subset X$,
\begin{equation}\label{sub-var-equ}
\int_{V}\Re(\alpha_{t}+\sqrt{-1}\omega)^{p}-\cot\left(\theta_0+\frac{\pi}{2}\right)\cdot\Im(\alpha_{t}+\sqrt{-1}\omega)^{p} > 0.
\end{equation}
then $\H_\omega\neq\emptyset$. Conversely, if $\H_\omega\neq\emptyset$ and $[\alpha]$ has hypercritical phase $\theta_0\in(0,\frac{\pi}{2})$, then \eqref{sub-var-equ} holds for any $p$-dimensional irreducible subvariety $V\subset X$.
\end{thm}
\begin{proof}

Suppose \eqref{sub-var-equ} holds for some $\Theta_{0}$-test family $\a_t$ where $\Theta_{0}=\frac\pi2+\theta_0$.  The non-emptiness of $\mathcal{H}_{\omega}$ follows from the argument of \cite[Theorem 1.3]{ChuLeeTakahashi2021} on the existence of dHYM solution under stability assumption, see also \cite[Section 5]{Chen2021}.  Since the proof is almost identical, we only point out the modifications.  As in  \cite[(7.2)]{ChuLeeTakahashi2021},  we consider the  twisted dHYM equation for $\a_{t,\varphi}=\a_t +\ddb\varphi(t)$:
\begin{equation}\label{twist-dHYM}
\Re(\a_{t,\varphi}+\sqrt{-1}\omega)^n-\cot\Theta_{0}\cdot \Im(\a_{t,\varphi}+\sqrt{-1}\omega)^n=c_t\omega^n
\end{equation}
where $c_t$ is the normalization constant so that their integral over $X$ coincides.  Define also the continuity path:
\begin{equation}
\mathcal{T}=\{ t\in [0,+\infty): \eqref{twist-dHYM}\; \text{admits a solution }\a_{t,\varphi}\in \Gamma_{\omega,\a_t,\Theta_{0},\tilde\Theta_0}\}
\end{equation}
where $\tilde \Theta_0\in (\Theta_{0},\pi)$ is some constant as in the proof of \cite[Theorem 1.3]{ChuLeeTakahashi2021}.  By assumption (ii), $c_t>0$ for all $t\in [0,+\infty)$.  The openness and closeness of $\mathcal{T}$ follows from the same argument. Since $c_0$ is strictly positive in this case (which is the only distinction from \cite{ChuLeeTakahashi2021}),  we obtain a $\varphi_0\in C^\infty(X)$ so that
\begin{equation}
\Re(\a_{\varphi_0}+\sqrt{-1}\omega)^n-\cot\Theta_{0}\cdot \Im(\a_{\varphi_0}+\sqrt{-1}\omega)^n=c_0\omega^n>0.
\end{equation}
In particular,  $Q_\omega(\a_{\varphi_0})\in (0,\theta_0+\frac\pi2)$ and hence $\varphi_0\in \mathcal{H}_{\omega}$.

\medskip

Conversely, if $\mathcal{H}_\omega\neq \emptyset$ and $[\a]$ has hypercritical phase $\theta_0\in(0,\frac\pi2)$.  Then there is $\varphi\in C^\infty(X)$ such that $Q_\omega(\a_\varphi)\in (0,\Theta_{0})$ where $\Theta_{0}=\frac\pi2+\theta_0<\pi$.  By the same argument of \cite[Lemma 2.3]{ChuLeeTakahashi2021} (see also \cite[Lemma 8.2]{CollinsJacobYau2020}), for any $p=1,2,\ldots,n$, we see that
\begin{equation}
\Im\left(e^{-\sqrt{-1}\Theta_{0}}(\alpha_{\vp}+\sqrt{-1}\omega)^{p}\right) < 0.
\end{equation}
We define the test family $\a_t=\a+t \omega$. Since $[\alpha_{\vp}]=[\alpha]=[\alpha_{0}]$,  for any $p$-dimensional subvariety $V\subset X$,
\begin{equation}
\int_{V}\Re(\alpha_{0}+\sqrt{-1}\omega)^{p}-\cot\Theta_{0}\cdot\Im(\alpha_{0}+\sqrt{-1}\omega)^{p} > 0.
\end{equation}
Since
\begin{equation}
\begin{split}
& \frac{d}{dt}\int_{V}\Re(\alpha_{t}+\sqrt{-1}\omega)^{p}-\cot\Theta_{0}\cdot\Im(\alpha_{t}+\sqrt{-1}\omega)^{p} \\
= {} & p\int_{V}\left(\Re(\alpha_{t}+\sqrt{-1}\omega)^{p-1}-\cot\Theta_{0}\cdot\Im(\alpha_{t}+\sqrt{-1}\omega)^{p-1}\right)\wedge \omega> 0.
\end{split}
\end{equation}
The assertion  follows. This completes the proof.
\end{proof}

\begin{rem}\label{CS question answer-2}
As in \cite[Corollary 1.4, Corollary 1.5]{ChuLeeTakahashi2021},  the stability condition  \eqref{sub-var-equ} in terms of test family can also be ensured by requiring: for some \K class $\chi$ in $X$ such that for any $p$-dimensional subvariety $V\subset X$ and $0\leq m\leq p$,
\begin{equation*}
\int_{V}\left\{\Re(\alpha+\sqrt{-1}\omega)^{p-m}-\cot\left(\theta_0+\frac{\pi}{2}\right)\cdot\Im(\alpha+\sqrt{-1}\omega)^{p-m}\right\}\wedge \chi^m> 0.
\end{equation*}
In particular, if $X$ is projective, then the above condition can be weaken as
\begin{equation}
\int_{V}\Re(\alpha+\sqrt{-1}\omega)^{p}-\cot\left(\theta_0+\frac{\pi}{2}\right)\cdot\Im(\alpha+\sqrt{-1}\omega)^{p} > 0.
\end{equation}
\end{rem}


\begin{thebibliography}{10}

\bibitem{Chen2021} Chen, G., {\sl The $J$-equation and the supercritical deformed Hermitian-Yang-Mills equation}, Invent. Math. 225 (2021), no. 2, 529--602.

\bibitem{ChuLeeTakahashi2021} Chu, J. Lee, M.-C., Takahashi, R., {\sl A Nakai-Moishezon type criterion for supercritical deformed Hermitian-Yang-Mills equation}, preprint, arXiv:2105.10725, to appear in J. Differential Geom.

\bibitem{CollinsJacobYau2020} Collins, T. C., Jacob, A., Yau, S.-T., {\sl $(1,1)$ forms with specified Lagrangian phase: a priori estimates and algebraic obstructions}, Camb. J. Math. 8 (2020), no. 2, 407--452.

\bibitem{CollinsShi2020} Collins, T. C., Shi, Y., {\sl Stability and the deformed Hermitian-Yang-Mills equation}, preprint, arXiv:2004.04831.

\bibitem{CollinsXieYau2018} Collins, T. C., Xie, D., Yau, S.-T., {\sl The deformed Hermitian-Yang-Mills equation in geometry and physics}, Geometry and physics. Vol. I, 69--90, Oxford Univ. Press, Oxford, 2018.

\bibitem{CollinsYau2018} Collins, T. C., Yau, S.-T., {\sl Moment maps, nonlinear PDE, and stability in mirror symmetry}, preprint, arXiv:1811.04824.

\bibitem{DatarPingali2021} Datar, V.; Pingali, V. P., {\sl A numerical criterion for generalised Monge-Amp\`ere equations on projective manifolds}, Geom. Funct. Anal. 31 (2021), no. 4, 767--814.

\bibitem{DemaillyPaun2004} Demailly, J.-P., P\u{a}un, M., {\sl Numerical characterization of the K\"ahler cone of a compact K\"ahler manifold}, Ann. of Math. (2) 159 (2004), no. 3, 1247--1274.

\bibitem{JacobSheu2020} Jacob, A., Sheu, N., {\sl The deformed Hermitian-Yang-Mills equation on the blowup of $\mathbb{P}^{n}$}, preprint, arXiv:2009.00651.

\bibitem{LeungYauZaslow2001} Leung, C., Yau, S.-T., Zaslow, E., {\sl From special Lagrangian to Hermitian-Yang-Mills via Fourier-Mukai transform}, Winter School on Mirror Symmetry, Vector Bundles and Lagrangian Submanifolds (Cambridge, MA, 1999), 209--225, AMS/IP Stud. Adv. Math., 23, Amer. Math. Soc., Providence, RI, 2001.

\bibitem{MMMS00} Mari\~no, M., Minasian, R., Moore, G., Strominger, A., {\sl Nonlinear instantons from supersymmetric $p$-branes}, J. High Energy Phys. {\bf 2000}, no. 1, Paper 5, 32 pp.

\bibitem{Song2020}Song, J., {\sl Nakai–Moishezon criterions for complex Hessian equations}, preprint, arXiv:2012.07956.

\bibitem{Solomon2013} Solomon, J. P., {\sl The Calabi homomorphism, Lagrangian paths and special Lagrangians}, Math. Ann. 357 (2013), no. 4, 1389--1424.

\bibitem{SYZ96} Strominger, A. Yau, S.-Y. Zaslow, E., {\sl Mirror symmetry is $T$-duality}, Nuclear Phys. B 479 (1996), no. 1-2, 243--259.

\bibitem{Thomas2001} Thomas, R. P., {\sl Moment maps, monodromy, and mirror manifolds, Symplectic geometry and mirror symmetry (Seoul, 2000)}, 467--498, World Sci. Publ., River Edge, NJ, 2001.

\end{thebibliography}
\end{document}